\documentclass[12pt]{article}
\usepackage{amsmath}
\usepackage{amssymb}
\usepackage{amsthm}
\newtheorem{lemma}{Lemma}
\newtheorem{conjecture}{Conjecture}
\title{Adjoining edges to $G\mathbin{\square}H$ to construct a \\
  minimal dominating set of size $\gamma(G)\gamma(H)$}
\author{Allan van Hulst \\
\small\tt allanvanhulst@protonmail.com}
\begin{document}
\maketitle
\begin{abstract}
\noindent For graphs $G,H$ it is possible to add
$(|V(G)|-\gamma(G))(|V(H)|-\gamma(H))$ edges to the Cartesian
product $G\mathbin{\square}H$ such that a minimal dominating
set $D$ of size $\gamma(G)\gamma(H)$ emerges. We hypothesize
that $D$ is also a minimum dominating set for the resulting
graph and show that this implies Vizing's conjecture. 
\end{abstract}
All graphs $G=(V(G),E(G))$ in this note are assumed to be simple
and without isolated vertices. A set $D\subseteq V(G)$ dominates $G$ 
if for all $v\in V(G)-D$ there exists a $u\in D$ such that $u\sim v$. 
The expression $\gamma(G)$ denotes the minimum cardinality of such a 
$D$. A dominating set $D$ is minimal if it does not have a strict subset 
that dominates $G$, which may include cases where $|D|>\gamma(G)$. The 
Cartesian product $G\mathbin{\square}H$ is a graph having $V(G)\times 
V(H)$ as its set of vertices. Adjacency in $G\mathbin{\square}H$ is 
defined as $(u,v)\sim(u',v')$ if and only if either (a) $u=u'$ and 
$v\sim v'$ in $H$ or (b) $v=v'$ and $u\sim u'$ in $G$. These 
definitions are standard in the literature on graphs and dominating 
sets \cite{cockayne,cockayne2}.

Vizing\footnote{I would like to thank Anna Kathekina for translating
parts of Vizing's original work for me.} postulated the following 
inequality over all $G$ and $H$ as a conjecture \cite{vizing,vizing2}:
\begin{equation}
\label{eqn:vizing}
\gamma(G)\gamma(H)\leq\gamma(G\mathbin{\square}H)
\end{equation}
which is a sharp bound if (\ref{eqn:vizing}) holds true. For more 
background information on Vizing's conjecture, survey articles are
available \cite{bresar,hartnell}. The objective of this note is to 
study a conceptual reversal of the inequality in (\ref{eqn:vizing}): it 
is shown that a number of edge-additions to $G\mathbin{\square}H$ creates
a minimal dominating set $D$ of size $\gamma(G)\gamma(H)$. If $D$ can 
also be shown to be a minimum dominating set for the edge-adjoined graph, 
Vizing's conjecture follows directly. For clarity, this conditional
implication refers to inequality $(\ref{eqn:vizing})$ and not to some
alternative form of Vizing's conjecture on these edge-adjoined graphs.

We now establish some further groundwork. The open neighborhood of a 
vertex $u\in V(G)$ is defined as $N(u)=\{v\in V(G)\mid u\sim v\}$.
A vertex $v\in V(G)-D$ is said to be a private neighbor of $u\in D$
if $u\sim v$ and $N(v)\cap D=\{u\}$. The set of private neighbors
of a $u\in D$ is denoted by $P(u)$. A dominating set $D\subseteq 
V(G)$ is said to be in canonical form if every $u\in D$ has at least 
one private neighbor in $G$. Canonical dominating sets of minimum size 
always exist, although these are not necessarily unique. 
\begin{lemma}
\label{lem:canonical}
[Adapted from \cite{bollobas}, page 7] For all graphs $G$ there exists 
a canonical dominating set $D\subseteq V(G)$ such that $|D|=\gamma(G)$.
\end{lemma}
\begin{proof}
Pick a minimum dominating set $D$ such that the number of edges in
the subgraph induced by $D$ is maximal. Let $u\in D$ and assume
towards a contradiction that $P(u)=\emptyset$. We distinguish
between two cases: (a) if there exists a $u'\in D$ such that
$u\sim u'$ then $D-\{u\}$ is a smaller dominating set and, (b)
if $N(u)\cap D=\emptyset$ then there must exist at least one
$v\in V(G)-D$ such that $u\sim v$ which makes $(D-\{u\})\cup\{v\}$
a dominating set inducing a subgraph with more edges.
\end{proof}
A dominating set $D$ in canonical form provides an insightful way 
to construct a surjective function from the non-dominated vertices
onto $D$, thereby following edges.
\begin{lemma}
\label{lem:surjective}
For all graphs $G$ and canonical dominating sets $D\subseteq V(G)$
there exists a surjection $F:V(G)-D\longrightarrow D$ such that
$F(v)=u$ implies $u\sim v$.
\end{lemma}
\begin{proof}
First, we create a surjective partial function by setting $F(v)=u$
for all $u\in D$ and $v\in P(u)$, Then, we can extend $F$ to a total
function by assigning any remaining vertices in $V(G)-D$ to an 
arbitrary neighbor in $D$. 
\end{proof}
We now move to the setting of the Cartesian product and extend the
surjective functions associated with the operands to a new function.
Let $D_G\subseteq V(G)$ and $D_H\subseteq V(H)$ be respective canonical 
dominating sets for graphs $G$ and $H$ derived by Lemma \ref{lem:canonical}. 
Additionally, let $F_G:V(G)-D_G \longrightarrow D_G$ and $F_H:V(H)-D_H\longrightarrow D_H$ 
be surjective functions as obtained from Lemma \ref{lem:surjective} and define
$S:V(G\mathbin{\square}H)-D_G\times D_H\longrightarrow D_G\times D_H$ as:
\begin{center}
\begin{math}
S(u,v)=\begin{cases}
       (F_G\,(u),v)      & \mathrm{if}\,\,u\in V(G)-D_G\,\,\mathrm{and}\,\,v\in D_H \\
       (u,F_H\,(v))      & \mathrm{if}\,\,u\in D_G\,\,\mathrm{and}\,\,v\in V(H)-D_H \\
       (F_G\,(u),F_H(v)) & \mathrm{if}\,\,u\in V(H)-D_G\,\,\mathrm{and}\,\,v\in V(H)-D_H. \\
       \end{cases}
\end{math}
\end{center}
Observe that $S$ is always total, univalent and surjective onto 
$D_G\times D_H$. We can now define an adjoint graph $A_{G,H}[F_G,F_H]$
as:
\begin{center}
\begin{math}
A_{G,H}[F_G,F_H]=(V(G\mathbin{\square}H),E(G\mathbin{\square}H)\cup E_S),
\end{math}
\end{center}
where $E_S$ is defined as:
\begin{center}
\begin{math}
E_S = \{\{(u,v),(u',v')\}\mid S(u,v)=(u',v'),\,(u,v)\in (V(G)-D_G)\times(V(H)-D_H)\}.
\end{math}
\end{center}
We may now prove an interesting property, notably:
\begin{lemma}
\label{lem:minimal}
$D_G\times D_H$ is a minimal dominating set for $A_{G,H}[F_G,F_H]$.
\end{lemma}
\begin{proof}
First, we show that $D_G\times D_H$ is a dominating set. Let
$(u,v)\in V(G\mathbin{\square}H)-D_G\times D_H$. If $u\in D_G$ and $v\in
V(H)-D_H$ then there exists an edge $v\sim v'$ for some $v'\in D_H$ and
an edge $(u,v)\sim(u,v')$ in $G\mathbin{\square}H$. The case for $u\in
V(G)-D_G$ and $v\in D_H$ is symmetric. If $u\in V(G)-D_G$ and $v\in V(H)-D_H$
then there exists a $(u',v')\in D_G\times D_H$ such that $S(u,v)=(u',v')$.

Then, we must show that $D_G\times D_H$ is minimal. Assume that $D=D_G\times D_H-
\{(u,v)\}$ for some $(u,v)\in D_G\times D_H$ and observe that none of the vertices
in $P(u)\times P(v)$ can be dominated by $D$.
\end{proof}
This leads to a conjecture stated in a form similar to Lemma \ref{lem:minimal}:
\begin{conjecture}
\label{con:minimum}
$D_G\times D_H$ is a minimum dominating set for $A_{G,H}[F_G,F_H]$.
\end{conjecture}
It is clear that any dominating set $D\subseteq V(G\mathbin{\square}H)$
for $G\mathbin{\square}H$, in particular when $|D|=\gamma(G\mathbin{\square}H)$,
must also dominate $A_{G,H}[F_G,F_H]$ since $G\mathbin{\square}H$ is a subgraph 
of $A_{G,H}[F_G,F_H]$ on the same set of vertices. Therefore, it is
sufficient to prove Conjecture \ref{con:minimum} to resolve Vizing's
conjecture as the result $\gamma(G)\gamma(H)=|D_G\times D_H|\leq|D|=
\gamma(G\mathbin{\square}H)$ is direct.

\end{document}